\documentclass{amsart}
\usepackage{stmaryrd}
\usepackage{silence}
\usepackage{upgreek}
\usepackage{amssymb}
\usepackage{faktor}
\usepackage[all,cmtip]{xy}
\usepackage{amsmath} 
\usepackage{mathrsfs}
\usepackage{tikz}
\usepackage{tikz-cd}
\usepackage{enumitem}
\usepackage{mathtools}
\usepackage[mathcal]{euscript}
\usepackage{graphicx}
\usepackage{url}
\usepackage{hyperref}
\usepackage[T1]{fontenc}
\usepackage{subfigure}

\usetikzlibrary{shapes.geometric}
\usetikzlibrary{decorations.markings}
\usetikzlibrary{patterns,decorations.pathreplacing}

\usepackage{ragged2e}

\newcommand{\Conf}{\mathrm{Conf}}

\newcommand{\id}{\mathrm{id}}

\newcommand{\pt}{\mathrm{pt}}

\definecolor{coloryellow}{RGB}{240,228,66}
\definecolor{colorskyblue}{RGB}{86,180,233}
\definecolor{colorvermillion}{RGB}{213,94,0}

\newcommand{\graphfont}{\mathsf}

\newcommand{\stargraph}[1]{\graphfont{S}_{#1}}
\newcommand{\graf}{\graphfont{\Gamma}}

\DeclareSymbolFont{sfletters}{OT1}{cmss}{m}{n}
\DeclareMathSymbol{\sTheta}{\mathord}{sfletters}{"02}

\theoremstyle{definition}
\newtheorem{definition}{Definition}[section]

\newtheorem{example}[definition]{Example}

\newtheorem{construction}[definition]{Construction}

\theoremstyle{plain}
\newtheorem{proposition}[definition]{Proposition}
\newtheorem{lemma}[definition]{Lemma}
\newtheorem{corollary}[definition]{Corollary}
\newtheorem{theorem}[definition]{Theorem}

\theoremstyle{remark}
\newtheorem{remark}[definition]{Remark}

\makeatletter
\@addtoreset{definition}{section}
\makeatother

\usepackage{marginnote}
    \DeclareFontFamily{U}{wncy}{}
    \DeclareFontShape{U}{wncy}{m}{n}{<->wncyr10}{}
    \DeclareSymbolFont{mcy}{U}{wncy}{m}{n}
    \DeclareMathSymbol{\Sha}{\mathord}{mcy}{"58}

\newsavebox{\foobox}

\title{On the stabilization of the topological complexity of graph braid groups}
\author{Ben Knudsen}
\email{b.knudsen@northeastern.edu}
\address{Department of Mathematics, Northeastern University, Boston, MA 02115, USA}

\begin{document}

\begin{abstract}
We establish a strong, geometric lower bound on the (sequential) topological complexity of the unordered configuration spaces of a general graph. As an application, we show that, for most graphs, the topological complexity eventually stabilizes at its maximal possible value, a direct analogue of a stability phenomenon in the ordered setting first conjectured by Farber. We estimate the stable range in terms of the number of trivalent vertices.
\end{abstract}

\maketitle

\section{Introduction}

Collision-free motion planning on graphs is a well-studied problem of significant theoretical and practical interest. One measure of the difficulty of this problem is provided by the quantity $\mathrm{TC}(\Conf_k(\graf))$, where $\mathrm{TC}$ denotes Farber's topological complexity \cite{Farber:TCMP} and $\Conf_k(\graf)$ the configuration space of $k$ ordered points on $\graf$. In 2005, Farber formulated a conjecture regarding this quantity \cite{Farber:CFMPG}, which we find most illuminating separated into three statements.\footnote{This division is motivated by the study of homological stability phenomena. The possibility of a connection between stability in that context and in ours remains unexplored.} As a matter of notation, we write $m(\graf)$ for the number of vertices of $\graf$ of valence at least $3$.
\begin{itemize}
\item {\bf Stability.} For some $k_0$, the sequence $\left(\mathrm{TC}(\Conf_k(\graf))\right)_{k\geq k_0}$ is constant.
\item {\bf Stable value.} The equality $\mathrm{TC}(\Conf_{k_0}(\graf))=2m(\graf)$ holds.
\item {\bf Stable range.} We may take $k_0= 2m(\graf)$.
\end{itemize}
What Farber conjectured, and what the author later proved \cite{Knudsen:TCPGBGSM} and generalized to the sequential or higher topological complexity $\mathrm{TC}_r$ \cite{Rudyak:HATC}, is that these three statements hold provided $\graf$ is connected with $m(\graf)\geq 2$.

Somewhat surprisingly, the obvious analogue of Farber's conjecture is false for the unordered configuration spaces $B_k(\graf):=\Conf_k(\graf)/\Sigma_k$; indeed, it is already false for linear trees \cite{LuetgehetmannRecio-Mitter:TCCSFAGBG}. So little is known about these spaces, however, that it has so far been totally unclear whether this failure is a failure of stability per se, or rather a failure of the stable value or stable range to coincide with their ordered counterparts.

Our main result shows that, for most graphs, the failure is merely that of a weaker stable range. As a matter of terminology, we say that a vertex $v$ is non-separating if its complement is connected.

\begin{theorem}\label{thm:stability}
Let $\graf$ be a connected graph with $m(\graf)\geq2$, and fix $r>0$. If $\graf$ has no non-separating trivalent vertices, then there exists $k_0$ such that, for all $k\geq k_0$, we have the equality \[\frac{1}{r}\mathrm{TC}_r(B_k(\graf))=m(\graf).\] We may take $k_0$ to be $2m(\graf)$ plus the number of trivalent vertices of $\graf$.
\end{theorem}

In particular, the exact analogue of Farber's conjecture holds in the unordered setting for graphs without trivalent vertices. It should be emphasized that the unordered setting was previously too mysterious even to warrant a conjecture.

The theorem will be an easy consequence of the following general lower bound.

\begin{theorem}\label{thm:estimate}
Let $\graf$ be a connected graph with $m(\graf)\geq2$, and fix $r>1$. Choose $c_i\in\mathbb{N}$ for $i\in\{0,1,2\}$ bounded above, respectively, by the number of vertices of valence at least $4$, the number of separating trivalent vertices, and the number of non-separating trivalent vertices of $\graf$. For any $k\geq 2(c_0+c_2)+3c_1$, we have the inequality \[\mathrm{TC}_r(B_k(\graf))\geq (r-2)\min\left\{\left\lfloor\frac{k}{2}\right\rfloor,\,m(\graf)\right\}+2(c_0+c_1)+c_2.\]
\end{theorem}

Prior to the techniques introduced in our work, the topological complexity of the unordered configuration spaces of graphs with positive first Betti number was essentially entirely unknown (see \cite{Scheirer:TCUCSCG,Hoekstra-Mendoza:BHTCCST} for prior work on trees). The obstruction to progress in this direction was ignorance of the relevant cohomology rings, which remains. We avoid this difficulty by working at the level of the fundamental group using results of \cite{FarberOprea:HTCAS}, following \cite{GrantLuptonOprea:NLBTCAS}, on the (sequential) topological complexity of aspherical spaces. By arguing geometrically, we further circumvent the fearsome combinatorial challenges posed by the standard Morse theoretic presentations of these groups \cite{FarleySabalka:DMTGBG,FarleySabalka:PGBG}. This geometric flavor is maintained through the use of configuration spaces with sinks, an approach inspired by \cite{LuetgehetmannRecio-Mitter:TCCSFAGBG} and already present implicitly in \cite{Knudsen:TCPGBGSM}.

\subsection{Relation to other work}
The strategy of this paper was adapted from the author's proof of Farber's conjecture \cite{Knudsen:TCPGBGSM}---see the survey \cite{Knudsen:FCB} for a more leisurely overview of these ideas.

As shown by Jankiewicz--Schreve, the conclusion of Theorem \ref{thm:stability} holds without the assumption on non-separating trivalent vertices \cite{JankiewiczSchreve:PFGIGBG}. Those authors largely follow our strategy, the added ingredient being an appeal to the work of \cite{CrispWiest:EGBSGRAAGBG} relating graph braid groups and right-angled Artin groups. In the end, the key fact is that the latter type of group is residually torsion-free nilpotent. Thus, despite beginning alike, the two arguments are of radically different flavor---ours direct and geometric, theirs indirect and algebraic.

The present work appeared as a preprint in substantially final form in February 2023, and the survey paper \cite{Knudsen:FCB} explaining its techniques and those of \cite{Knudsen:TCPGBGSM} in February 2024. While the survey was published later in 2024, no referee report regarding the present work appeared until December 2025. The work of Jankiewicz--Schreve appeared as a preprint in April 2024 and was published in October 2025.

\subsection{Future directions} Our work prompts a number of questions.
\begin{itemize}
\item {\bf Stable range.} How does the topology of $\graf$ influence the stable range? In the ordered setting, there are examples with $k_0=2$ and $m(\graf)$ arbitrarily large \cite{GonzalezGonzalez:ADGFCACSTPCG}, but a systematic understanding remains elusive.
\item {\bf Unstable values.} What can be said about the topological complexity of $B_k(\graf)$ and $\Conf_k(\graf)$ for $k<k_0$?
\item {\bf Monotonicity.} For a graph $\graf$, is the function $k\mapsto\mathrm{TC}(B_k(\graf))$ non-decreasing? As shown in \cite{LuetgehetmannRecio-Mitter:TCCSFAGBG}, the corresponding question in the ordered setting has an affirmative answer. In either setting, does monotonicity hold with $\graf$ replaced by a general topological space?
\end{itemize}

\subsection{Conventions and notation} Given an equivalence relation $\pi$ on a set $I$, we write $[i]$ for the equivalence class of $I\in I$ and $I/\pi$ for the set of all such equivalence classes. A partition of a number $n$ is a tuple $p=(p_i)_{i\in I}$ with $I$ a finite set such that $p_i\geq0$ for $i\in I$ and $\sum_{i\in I}p_i=n$. In particular, partitions have distinguishable blocks, some of which may be trivial.

We write $\mathrm{cd}_R(G)$ for the cohomological dimension of the group $G$ over the commutative ring $R$, which is to say the minimal length of a projective resolution of the trivial module $R$ over the group ring $R[G]$. We set $\mathrm{cd}(G)=\mathrm{cd}_\mathbb{Z}(G)$. We denote units of groups generically by $1$.

Given a path $\gamma$ in a space $X$, we denote its path homotopy class by $[\gamma]$ and its change-of-basepoint isomorphism by $\hat\gamma:\pi_1(X,\gamma(0))\to \pi_1(X,\gamma(1))$. When there is no risk of confusion, we omit basepoints from our notation for fundamental groups.

A graph is a finite CW complex $\graf$ of dimension at most $1$. The degree or valence $d(v)$ of a vertex $v$ of $\graf$ is the number of connected components of its complement in a sufficiently small open neighborhood. A vertex $v$ is called essential if $d(v)\geq 3$ and separating if $\graf\setminus\{v\}$ is disconnected. Edges of graphs are always open cells; when referring to the closed cell corresponding to the edge $e$, we write $\bar e$.

We use the ``reduced'' convention for topological complexity, in which $\mathrm{TC}(\pt)=0$. In dealing with the sequential topological complexity $\mathrm{TC}_r$, our convention is that $\mathrm{TC}_1=\mathrm{cat}$, the Lusternik--Schnirelmann category.

\subsection*{Acknowledgements} This work benefited from conversations with Byunghee An, Teresa Hoekstra--Mendoza, and Jes\'{u}s Gonz\'{a}lez, as well as from the referee's careful reading. 

\subsection*{Funding declaration} This work was supported by NSF grant DMS-1906174.

\section{Configurations with controlled collisions}

It will be convenient to work with a mild generalization of classical configuration spaces, in which particles are permitted to collide in a fixed subspace. The case of interest for us, in which the background space is a graph and the subspace in question a set of vertices, was introduced in \cite{ChettihLuetgehetmann:HCSTL} and further studied in \cite{Ramos:CSGCPC}.

\begin{definition}
Let $X$ be a topological space and $S\subseteq X$ any subset. The ordered \emph{configuration space} of $k$ points in $X$ with \emph{collisions in $S$} is \[\Conf_k(X,S)=\{(x_1,\ldots, x_k)\in X^k: x_i=x_j\implies i=j \text{ or } x_i\in S\}.\] The \emph{unordered} configuration space is $B_k(X,S)=\Conf_k(X,S)/\Sigma_k$.
\end{definition}

In the classical case $S=\varnothing$, we simply write $\Conf_k(X)$ and $B_k(X)$.

\begin{remark}
The reader is urged not to conflate collision in a subspace with \emph{annihilation} in a subspace, as considered classically (for example) in \cite{Bodigheimer:SSMS}.
\end{remark}

In order to indicate the functoriality of these constructions, we introduce the following concepts.

\begin{definition}
We say that a map of pairs $f:(X,S)\to (Y,T)$ is a \emph{near monomorphism} if $f|_{f^{-1}(Y\setminus T)}$ is a monomorphism. A \emph{near monotopy} is a homotopy through near monomorphisms.
\end{definition}

The proof of the following simple fact is essentially immediate from the definitions and left to the reader.

\begin{proposition}\label{prop:functoriality}
Let $f:(X,S)\to (Y,T)$ be a near monomorphism.
\begin{enumerate}
\item The dashed filler exists in the following commutative diagram of $\Sigma_k$-spaces:
\[
\xymatrix{
\Conf_k(X,S)\ar[d]_-{\subseteq}\ar@{-->}[rr]^-{\Conf_k(f)}&&\Conf_k(Y,T)\ar[d]^-{\subseteq}\\
X^k\ar[rr]^-{f^k}&&Y^k
}
\]
\item If $f$ is near monotopic to $g$, then $\Conf_k(f)$ is $\Sigma_k$-equivariantly homotopic to $\Conf_k(g)$.
\end{enumerate}
\end{proposition}

Thus, the constructions $\Conf_k$ and $B_k$ extend canonically to homotopy functors from the category of near monomorphisms between pairs of spaces and near monotopies among them. In particular, we obtain functoriality for quotient maps, as the following example shows.

\begin{example}\label{example:quotient}
Given a pair $(X,S)$ and a subset $A\subseteq X$ containing $S$, the projection $(X,S)\to (X/A, A/A)$ is a near monomorphism. More generally, the same holds for the projection to the quotient of $X$ by an equivalence relation. 
\end{example}

We turn now to the case of greatest present interest.

\begin{definition}
A \emph{graph with sinks} is a pair $(\graf, S)$, where $\graf$ is a graph and $S$ a subset of the vertices of $\graf$. The elements of $S$ are called \emph{sinks}.
\end{definition}

Although we are primarily interested in the classical configuration spaces of graphs, it will be convenient at times to allow collisions at sink vertices. In this way, we gain access to functoriality for quotients as in Example \ref{example:quotient}. When the set of sinks is understood from context, we will permit ourselves the (highly abusive) abbreviation $B_k(\graf,S)=B_k(\graf)$.

The fundamental fact about configuration spaces of graphs is that they classify their fundamental groups, the so-called graph braid groups.

\begin{theorem}[{\cite{Abrams:CSBGG,Swiatkowski:EHDCSG,Ghrist:CSBGGR}}]\label{thm:asphericity}
For any graph $\graf$ and $k\geq0$, the space $B_k(\graf)$ is aspherical.
\end{theorem}

\begin{remark}
The author is unaware of a reference establishing the asphericity of the configuration spaces of an arbitrary graph with sinks (although see \cite{Ramos:CSGCPC} for the case in which every vertex is a sink). He presumes that the generalization of the cubical complex of \cite{Swiatkowski:EHDCSG} exhibited in \cite{ChettihLuetgehetmann:HCSTL} can be shown to be CAT(0) through a straightforward adaptation of classical arguments. The enthusiastic reader is invited to check the details.
\end{remark}

In working with these spaces, we make use of a certain deformation retract. Fix a graph $\graf$, a subspace $S\subseteq\graf$, and a vertex $v$. Write $B_k(\graf,S)_v\subseteq B_k(\graf,S)$ for the subspace in which at most one particle lies in the open star of $v$.

\begin{proposition}\label{prop:deformation retract}
Let $(\graf,S)$ be a graph with sinks and $v$ a non-sink vertex. The inclusion $B_k(\graf,S)_v\subseteq B_k(\graf, S)$ is a deformation retract provided either $k\leq2$ or every non-sink vertex adjacent to $v$ has valence $2$ and no multiple edges.
\end{proposition}
\begin{proof}
For each sink $w$ adjacent to $v$, subdivide each edge from $w$ to $v$ by adding a bivalent vertex, denoting the resulting graph by $\graf'$. Let $A_w$ denote the union of the (closed) edges from $w$ to each of the new vertices, and write $A$ for the union of the $A_w$ and $S$. The obvious map of pairs $(\graf,S)\to (\graf',A)$ gives rise to the following commutative diagram: \[\xymatrix{
B_k(\graf,S)_v\ar[d]_-\subseteq\ar[r]&B_k(\graf',A)_v\ar[d]^-\subseteq\\
B_k(\graf,S)\ar[r]&B_k(\graf',A).
}\] The horizontal maps are topological embeddings, so we may regard them as inclusions of subspaces. The obvious deformation of $A$ onto $S$ gives rise to a near monotopy supplying a deformation retraction of the bottom inclusion, and, by inspection, this deformation retraction preserves the subspaces in question. Thus, in order to conclude that the lefthand map is the inclusion of a deformation retract, it suffices to note that the righthand map is so by the construction of \cite[Lem. 2.0.1]{AgarwalBanksGadishMiyata:DCCSG}.
\end{proof}

\begin{remark}
The corresponding result holds for ordered configuration spaces, with the same proof. In the classical case $S=\varnothing$, this result is was mistakenly asserted in \cite[Prop. 2.2]{Knudsen:TCPGBGSM} without hypotheses. In this generality, it is false---consider, for example, the cycle with two edges. As shown here, this difficulty can always be avoided by subdividing.
\end{remark}

\begin{remark}
The proof of Proposition \ref{prop:deformation retract}, as well as the results below on configuration spaces of local graphs and maps among them, could equally be approached through the cubical model introduced in \cite{Swiatkowski:EHDCSG} and generalized in \cite{ChettihLuetgehetmann:HCSTL}. We choose to give a self-contained treatment.
\end{remark}

\section{Edge and sink stabilization}

In this section, we describe ``stabilization'' mechanisms increasing the number of particles occupying a fixed edge or sink in a continuous manner. See \cite{AnDrummondColeKnudsen:ESHGBG} for more on edge stabilization.

\begin{definition}\label{def:stabilizations}
Fix a graph $\graf$ with sinks $S$, a (parametrized) edge $e$ of $\graf$, and a sink $s\in S$.
\begin{enumerate}
\item The \emph{edge stabilization} map $\sigma_e:B_k(\graf,S)\to B_{k+1}(\graf,S)$ associated to $e$ is the map whose value on a configuration $x$ is obtained by replacing the subconfiguration $x\cap \overline{e}=\{x_i\}_{i=1}^\ell$, ordered according to the parametrization of $e$, with the collection $\left\{\frac{1}{2}(x_{i+1}-x_{i})\right\}_{i=0}^\ell$ of averages, where $x_0=0$ and $x_{\ell+1}=1$ are the endpoints of $e$.
\item The \emph{sink stabilization} map $\sigma_s:B_k(\graf,S)\to B_{k+1}(\graf,S)$ associated to $v$ is the map obtained by adding a particle at $s$.
\end{enumerate}
\end{definition}

Establishing the continuity of $\sigma_e$ is an easy exercise using the pasting lemma, and the continuity of $\sigma_s$ is essentially immediate. Changing the parametrization of $e$ changes $\sigma_e$ only up to homotopy.

We close with the following simple observation.

\begin{lemma}\label{lem:stabilization comparison}
Fix a graph $\graf$ with sinks $S$, a (parametrized) edge $e$ of $\graf$, and a sink $s\in S$. If $e$ is incident on $s$, then $\sigma_e$ and $\sigma_s$ are homotopic.
\end{lemma}
\begin{proof}
Write $\graf'$ for the graph obtained by subdividing $e$ once. We denote the corresponding edges of $\graf'$ by $e'$ and $e''$, where $e'$ is incident on $s$, and we write $A$ for the union of $S$ with the closure of $e'$. Consider the diagram
\[
\xymatrix{B_k(\graf, S)\ar[d]_-{\sigma_{e}}\ar@{=}[r]^-\sim&B_k(\graf', S)\ar@{-->}[d]\ar[r]&B_k(\graf', A)\ar[d]^-{\sigma_{e'}}\ar[r]&B_k(\graf,S)\ar[d]^-{\sigma_s}\\
B_{k+1}(\graf,S)\ar@{=}[r]^-\sim&B_{k+1}(\graf', S)\ar[r]&B_{k+1}(\graf',A)\ar[r]&B_{k+1}(\graf,S)
}\] where the dashed map is determined by commutativity of the lefthand square; the map $\sigma_{e'}$ is defined as in Definition \ref{def:stabilizations}(1); the homeomorphisms are induced by the homeomorphism $(\graf,S)\cong(\graf',S)$; the middle pair of horizontal maps is induced by the inclusion $S\subseteq A$, and the rightmost pair of horizontal maps is induced by the map $(\graf',A)\to (\graf,S)$ collapsing the closure of $e'$ to $s$. The lefthand square commutes by fiat, the righthand square commutes by inspection, and the middle square commutes up to a straight line homotopy. Since each horizontal composite is homotopic to the respective identity map by Proposition \ref{prop:functoriality}, the claim follows.
\end{proof}

\section{Local graphs}

We now introduce a family of atomic graphs (with sinks) functioning as the building blocks of our later arguments.

\begin{definition}\label{def:local graph}
Let $I$ be a finite set and $\pi$ an equivalence relation on $I$. The \emph{local graph} on $\pi$ is the graph $\graphfont{\Lambda}(\pi)$ with set of vertices $\{v_0\}\sqcup\{v_{[i]}\}_{i\in I/\pi}$, in which $v_0$ and $v_{[i]}$ share a set of edges indexed by the equivalence class $[i]$. We regard every vertex of $\graphfont{\Lambda}(\pi)$ as a sink except for $v_0$.
\end{definition}

Write $\stargraph{I}$ for the cone on the discrete space $I$ with its canonical (parametrized) graph structure and $\partial=\partial\stargraph{I}$ for its set of univalent vertices. In the case $I=\{1,\ldots, n\}$, we write $\stargraph{n}$. We refer to $\stargraph{n}$ as the \emph{star graph} with $n$ edges.

\begin{example}\label{example:discrete}
The local graph on the discrete partition of $I$ is the graph with sinks $(\stargraph{I},\partial)$. We denote this local graph by $\widetilde{\graphfont{S}}_I$.\end{example}

\begin{example}\label{example:indiscrete}
The local graph on the indiscrete partition of $I$ is the graph with sinks $(\stargraph{I}/\partial, \partial/\partial)$. We denote this local graph (abusively) by $\stargraph{I}/\partial$.
\end{example}

For our purposes, the interest of local graphs is as follows (cf. \cite{LuetgehetmannRecio-Mitter:TCCSFAGBG}).

\begin{construction}\label{construction:local quotient}
Given a vertex $v$ of the graph $\graf$ (sinkless, without self-loops), regard two edges at $v$ as equivalent if they lie in the same component of $\graf\setminus\{v\}$. Writing $\pi_v$ for the resulting equivalence relation on the set of edges at $v$, and making the abbreviation $\graphfont{\Lambda}(v):=\graphfont{\Lambda}(\pi_v)$, there is an obvious quotient map $q_v:\graf\to \graphfont{\Lambda}(v)$ obtained by mapping any point outside the closed star of $v$ to the vertex of $\graphfont{\Lambda}(v)$ indexed by its connected component in $\graf\setminus \{v\}$.
\end{construction}

The map on configuration spaces induced by the various $q_v$ will be a key player in our main argument. As we now show, the configuration spaces of local graphs are themselves graphs, up to homotopy, and thus easily understood.

\begin{construction}
Fixing an equivalence relation $\pi$ on $I$ and $k\geq0$, let $\graphfont{\Lambda}_k(\pi)$ be the graph specified as follows.
\begin{enumerate}
\item For every partition $p=(p_{[i]})_{[i]\in I/\pi}$ of $k$ or of $k-1$, there is a vertex $v_p$.
\item\label{edge description} If $p'$ is obtained from $p$ by replacing $p_{[i]}$ by $p_{[i]}-1$, then $v_p$ and $v_{p'}$ share a set of edges indexed by the equivalence class $[i]$.
\end{enumerate}
We define a function $f_k:{\graphfont{\Lambda}}_k(\pi)\to B_k(\graphfont{\Lambda}(\pi))$ as follows.
\begin{enumerate}[resume]
\item The configuration $f_{k}(v_p)$ comprises $p_{[i]}$ particles located at $v_{[i]}$. If $p$ is a partition of $k-1$, then $f_{k}(v_p)$ also comprises a particle located at $v_0$.
\item If $p$ and $p'$ are as in (2), then the image under $f_{k}$ of the edge indexed by $j\in[i]$ is the path in which the particle at $v_0$ moves linearly onto $v_{[i]}$ along the edge of $\graphfont{\Lambda}(\pi)$ indexed by $j$.\footnote{The graphs $\graphfont{\Lambda}(\pi)$ and $\graphfont{\Lambda}_k(\pi)$ may be constructed as quotients of a disjoint union of singletons and unit intervals in a way that is canonical up to orientation. As such, their edges are canonically parametrized up to orientation, and it is in this sense that we intend the term ``linear.''}
\end{enumerate}
\end{construction}

Note that $f_{1}$ is a homeomorphism. In general, we have the following.

\begin{proposition}\label{prop:local graph}
The function $f_k:{\graphfont{\Lambda}}_k(\pi)\to B_k(\graphfont{\Lambda}(\pi))$ is a homeomorphism onto $B_k(\graphfont{\Lambda}(\pi))_{v_0}$. In particular, $f_k$ is a homotopy equivalence.
\end{proposition}
\begin{proof}
Continuity is immediate from the pasting lemma, and $f_k$ is obviously injective. Since the source is compact and the target Hausdorff, it follows that $f_k$ is a homeomorphism onto its image, which is contained in $B_k(\graphfont{\Lambda}(\pi))_{v_0}$ by construction. From the definition, a configuration lies in $B_k(\graphfont{\Lambda}(\pi))_{v_0}$ if and only if $k-1$ of its particles are distributed among the sink vertices. It is immediate from this description that every such configuration lies in the image of $f_{k}$, and the first claim follows. The second claim then follows from Proposition \ref{prop:deformation retract}.
\end{proof}

We will use this result to deduce a number of useful facts relating $\stargraph{I}$ and its cousins for various $I$. To begin, we observe two families of canonical maps among them, in addition to edge and sink stabilizations. First, given a subset $I\subseteq J$, the inclusion of $I$ induces a topological embedding on cones, from which we derive canonical maps $B_k(\stargraph{I})\to B_k(\stargraph{J}),$ and similarly for the local graphs of Examples \ref{example:discrete} and \ref{example:indiscrete}. We refer to these maps by the word \emph{embedding}. Second, the identity $(\stargraph{I}, \varnothing)\to ({\graphfont{S}}_I,\partial)$ induces a map $B_k(\stargraph{I})\to B_k(\widetilde{\graphfont{S}}_I)$, which we refer to by the word \emph{reduction}.

Before proceeding, we recall the following elementary fact.

\begin{lemma}\label{lem:subgraph retract}
A connected subgraph is a topological retract.
\end{lemma}
\begin{proof}
Let $\graphfont{\Delta}_0\subseteq\graphfont{\Delta}$ be a connected subgraph. Since the components of a graph are always topological retracts, and since retractions compose, we may assume without loss of generality that $\graphfont{\Delta}$ is itself connected. We proceed by induction on the number of edges of $\graphfont{\Delta}$ not contained in $\graphfont{\Delta}_0$, the base case being trivial. Choose an edge $e$ sharing a vertex with $\graphfont{\Delta}_0$, but not contained therein, and write $\graphfont{\Delta}'=\graphfont{\Delta}_0\cup \overline e$. Since $e$ shares a vertex with $\graphfont{\Delta}_0$, the subgraph $\graphfont{\Delta}'$ is also connected, hence a topological retract by induction; therefore, it suffices to show that $\graphfont{\Delta}_0$ is a topological retract of $\graphfont{\Delta}'$. If $e$ shares only one vertex with $\graphfont{\Delta}_0$, then a retraction is obtained by collapsing $\overline e$ to this one vertex; otherwise, a retraction is obtained by choosing any path in $\graphfont{\Delta}_0$ between the vertices of $e$.
\end{proof}

\begin{corollary}\label{cor:homotopy retracts}
For finite sets $I$ and $J$, the following maps admit retractions up to homotopy.
\begin{enumerate}
\item The embedding $B_k(\stargraph{I})\to B_k(\stargraph{J})$ for $I\subseteq J$.
\item The embedding $B_k(\widetilde{\graphfont{S}}_I)\to B_k(\widetilde{\graphfont{S}}_J)$ for $I\subseteq J$.
\item The embedding $B_k(\stargraph{I}/\partial)\to B_k(\stargraph{J}/\partial)$ for $I\subseteq J$.
\item Any sink stabilization $B_k(\graphfont{\Lambda}(\pi))\to B_{k+1}(\graphfont{\Lambda}(\pi))$.
\end{enumerate}
The following maps are homotopy equivalences.
\begin{enumerate}[resume]
\item The sink stabilization $B_k(\stargraph{I}/\partial)\to B_{k+1}(\stargraph{I}/\partial)$.
\item The reduction $B_k(\stargraph{I})\to B_k(\widetilde{\graphfont{S}}_I)$.
\end{enumerate}
\end{corollary}
\begin{proof}
The second and sixth claim imply the first. By inspection, through Propositions \ref{prop:deformation retract} and \ref{prop:local graph}, the second, third, and fourth maps are equivalent (up to homotopy) to the inclusions of connected subgraphs, to which Lemma \ref{lem:subgraph retract} applies, and the fifth is equivalent to an isomorphism of graphs. For the sixth, write $\stargraph{I}'$ for the graph obtained by subdividing each edge of $\stargraph{I}$ once, and consider the following diagram:
\[\xymatrix{
B_k(\stargraph{I}')_{v_0}\ar@{-->}[d]_-g\ar[r]^-{\subseteq}&B_k(\stargraph{I}')\ar[r]^-\cong\ar[d]&B_k(\stargraph{I})\ar[d]\\
\graphfont{\Lambda}_k(\pi_\delta)\ar[r]^-{f_k}&B_k(\widetilde{\graphfont{S}}_I)\ar@{=}[r]&B_k(\widetilde{\graphfont{S}}_I),\\
}\] where $\pi_\delta$ denotes the discrete partition of $I$, the map $g$ is determined by requiring that the lefthand subdiagram commute, the top righthand arrow is induced by the canonical homeomorphism, and the middle arrow is induced by the map that collapses each leaf of $\stargraph{I}'$ to a point. The righthand subdiagram commutes up to homotopy by Proposition \ref{prop:functoriality}(2), so it suffices by Propositions \ref{prop:deformation retract} and \ref{prop:local graph} to show that $g$ is a weak equivalence. For this purpose, we apply \cite[16.24]{Gray:HT} to the open cover of $\graphfont{\Lambda}_k(\pi_\delta)$ given by the collection of open edges and open stars of vertices, recalling the elementary fact that the unordered configuration spaces of an interval are contractible.
\end{proof}

\section{Motion planning in aspherical spaces}

Farber's topological complexity is a numerical measure of the difficulty of the motion planning problem in a given background space \cite{Farber:TCMP}. This invariant belongs to an infinite family of invariants, denoted $\mathrm{TC}_r$ for $r>0$, which pertain to sequential motion planning \cite{Rudyak:HATC}. For definitions and further elaboration, the reader is invited to consult the references, as our interaction with these concepts will be mediated exclusively through a result constraining their behavior in the aspherical context. 

In order to formulate this group theoretic result, we introduce the following piece of group theoretic language.

\begin{definition}\label{def:disjoint conjugates}
Let $G$ be a group and $H_0,H_1\leq G$ subgroups. We say that $H_0$ and $H_1$ have \emph{disjoint conjugates} if $ghg^{-1}\in H_1$ for $h\in H_0$ and $g\in G$ if and only if $h=1$. We say that two homomorphisms with a common target have disjoint conjugates if their images do so.
\end{definition}

Equivalently, two subgroups have disjoint conjugates if every conjugate of the first has trivial intersection with the second. It is not hard to see that the two subgroups play symmetric roles in this definition.

The relevance of these ideas to our purpose lies in the following result.

\begin{theorem}[{\cite[Thm. 2.1]{FarberOprea:HTCAS}}]\label{thm:lower bound}
Let $G$ be a group, and fix $r\geq 2$. If the subgroups $H_0,H_1\leq G$ have disjoint conjugates, then \[\mathrm{TC}_r(BG)\geq \mathrm{cd}(H_0\times H_1\times G^{r-2}).\]
\end{theorem}

We close by recording a few simple techniques for working with this concept.

\begin{lemma}\label{lem:disjointness homomorphism}
Let $\psi:G\to K$ be a group homomorphism and $H_0,H_1\leq G$ subgroups.
\begin{enumerate}
\item Suppose that $\psi|_{H_0}$ is injective. If $\psi(H_0)$ and $\psi(H_1)$ have disjoint conjugates, then so do $H_0$ and $H_1$.
\item Suppose that $\psi$ admits a retraction. If $H_0$ and $H_1$ have disjoint conjugates, then so do $\psi(H_0)$ and $\psi(H_1)$.
\end{enumerate}
\end{lemma}
\begin{proof}
For the first claim, suppose that $ghg^{-1}\in H_1$ for $h\in H_0$ and $g\in G$. Since $\psi$ is a homomorphism, it follows that $\psi(g)\psi(h)\psi(g)^{-1}\in \psi(H_1)$. Since $\psi(h)\in \psi(H_0)$, our disjointness assumption implies that $\psi(g)\psi(h)\psi(g)^{-1}=1$, whence $\psi(h)=1$. The claim now follows from our injectivity assumption.

For the second claim, let $\rho:K\to G$ be a homomorphism with $\rho\circ \psi=\id_G$, and suppose that $k\psi(h)k^{-1}\in \psi(H_1)$ for $h\in H_0$ and $k\in K$. From our assumptions on $\rho$, it follows that $\rho(k)h\rho(k)^{-1}\in H_1$. Our disjointness assumption implies that $h=1$, whence $\psi(h)=1$.
\end{proof}

\begin{corollary}\label{cor:kernel}
Let $G$ be a group and $H_0,H_1\leq G$ subgroups. If there is a homomorphism $\psi:G\to K$ with $\psi|_{H_0}$ injective and $\psi|_{H_1}$ trivial, then $H_0$ and $H_1$ have disjoint conjugates.
\end{corollary}
\begin{proof}
Our triviality assumption implies that $\psi(H_1)$ is trivial, so $\psi(H_0)$ and $\psi(H_1)$ have disjoint conjugates, and Lemma \ref{lem:disjointness homomorphism}(1) yields the result.
\end{proof}

\section{Local calculation}

This section is dedicated to the computations that will, after some elaboration, imply the main theorems. We begin with the well known observation that, within the configuration space $B_2(\stargraph{3})$, there is the loop given by the threefold concatenated path \[\epsilon=\{\underline v_1, e_{23}\}\star \{e_{12}, \underline v_3\}\star \{\underline v_2, e_{31}\},\] where $v_i$ denotes the $i$th vertex, an underscore indicates a constant path, and $e_{ij}:[0,1]\to \stargraph{3}$ is the unique piecewise linear path from $v_i$ to $v_j$ with $e_{ij}(\frac{1}{2})=v_0$, the essential vertex. Note that $\epsilon$ is a loop based at the configuration $\{v_1, v_2\}$ (here and above we use braces to indicate unordered configurations).

It is a standard fact that $\epsilon:S^1\to B_2(\stargraph{3})$ is a homotopy equivalence. More precisely, we have the following (see \cite[Lem. 4.2]{Knudsen:TCPGBGSM}, for example).

\begin{lemma}\label{lem:local homeomorphism}
The map $\epsilon$ factors through an embedding $S^1\to B_2(\stargraph{3})_{v_0}$ as a deformation retract.
\end{lemma}

Given $n\geq3$ and $0\leq a\leq n-3$, there is a unique piecewise linear embedding $\stargraph{3}\to \stargraph{n}$ sending the $i$th edge of the source to the $(i+a)$th edge of the target. We write $\iota_a:B_2(\stargraph{3})\to B_2(\stargraph{n})$ for the induced map on configuration spaces. In the same way, we obtain maps $\tilde\iota_a:B_2(\stargraph{3}/\partial)\to B_2(\stargraph{n}/\partial)$.

\begin{lemma}\label{lem:local calculation}
Fix $n\geq 3$ and $k\geq2$. For distinct $0\leq a\leq n-3$, the composite homomorphisms \[\pi_1(B_2(\stargraph{3}))\xrightarrow{(\iota_a)_*}\pi_1(B_2(\stargraph{n}))\to \pi_1(B_2(\stargraph{n}/\partial))\to \pi_1(B_k(\stargraph{n}/\partial))\] have disjoint conjugates, where the unmarked arrows are induced by the quotient map and sink stabilization, respectively. Moreover, each of these composites is injective.
\end{lemma}
\begin{proof}
We treat in detail only the case $a\in\{0,1\}$, which is the case we use below. Modifying the proof in the remaining cases is straightforward and left to the reader.

By Corollary \ref{cor:homotopy retracts}(5), we may take $k=2$. In light of the commutative diagram
\[\xymatrix{
B_2(\stargraph{3})\ar[dr]^-{\iota_a}\ar[ddd]\ar[rr]^-{\iota_a}&&B_2(\stargraph{n})\ar[ddd]\\
&B_2(\stargraph{4})\ar[d]\ar[ur]\\
&B_2(\stargraph{4}/\partial)\ar[dr]\\
B_2(\stargraph{3}/\partial)\ar[ur]^-{\tilde\iota_a}\ar[rr]^-{\tilde\iota_a}&&B_2(\stargraph{n}/\partial),
}\] together with Corollary \ref{cor:homotopy retracts}(3) and Lemma \ref{lem:disjointness homomorphism}(2), we may take $n=4$. Finally, from the same diagram, it suffices to show that the composite homomorphisms \[\pi_1(B_2(\stargraph{3}))\xrightarrow{}\pi_1(B_2(\stargraph{3}/\partial))\xrightarrow{(\tilde\iota_a)_*}\pi_1(B_2(\stargraph{4}/\partial))\] have disjoint conjugates for $a\in\{0,1\}$.

Denoting by $\gamma_{i}$ a piecewise linear loop in $B_2(\stargraph{n}/\partial)$ in which a single particle moves from the sink vertex to the essential vertex along the $i$th edge and back along the $(i+1)$st, it follows from Proposition \ref{prop:local graph} and direct calculation that $\pi_1(B_2(\stargraph{n}/\partial))$ is freely generated by $\{[\gamma_{1}],\ldots,[\gamma_{n-1}]\}$, and that the composite in question is the inclusion of the subgroup generated by the commutator $\left[[\gamma_{1}],[\gamma_{2}]\right]$ or $\left[[\gamma_{2}],[\gamma_{3}]\right]$, according to the value of $a$. Indeed, the map $B_2(\stargraph{3})\to B_2(\stargraph{3}/\partial)$ sends $\{\underline v_1, e_{23}\}$ to $\gamma_2$, and so on.

 The claim now follows by applying Corollary \ref{cor:kernel} to the homomorphism $\psi:\pi_1(B_2(\stargraph{4}/\partial))\to \pi_1(B_2(\stargraph{3}/\partial))$ determined by the equation \[
\psi\left([\gamma_i]\right)=
\begin{cases}
[\gamma_i]&\quad i\in\{1,2\}\\
1&\quad i=3
\end{cases} 
\] and the universal property of the free group.
\end{proof}

\begin{lemma}\label{lem:other local calculation}
Fix $k\geq3$ and an equivalence relation $\pi$ on $\{1,2,3\}$. If $\pi$ does not identify $1$ and $2$, then, for $a\in\{0,1\}$, the homomorphisms 
\[
\pi_1(B_2(\stargraph{3}))\xrightarrow{\sigma_{e_{a+1}}}\pi_1(B_3(\stargraph{3}))\to \pi_1(B_3(\graphfont{\Lambda}(\pi)))\to \pi_1(B_k(\graphfont{\Lambda}(\pi)))
\] have disjoint conjugates, where the unmarked arrows are induced by the quotient map and any composite of sink stabilizations, respectively. Moreover, each of these composites is injective.
\end{lemma}
\begin{proof}
The argument is similar to that of Lemma \ref{lem:local calculation}, albeit simpler, so we will leave some details to the reader. By Corollary \ref{cor:homotopy retracts}(4) and Lemma \ref{lem:disjointness homomorphism}(2), we may take $k=3$. By symmetry, there are two cases, namely that of the discrete equivalence relation and that of the equivalence relation $\pi$ identifying $2$ and $3$, and the former follows from the latter by means of Lemma \ref{lem:disjointness homomorphism}(1). In the latter case, it follows from Proposition \ref{prop:local graph} and direct calculation that $\pi_1(B_3(\graphfont{\Lambda}(\pi)))$ is free on three generators, and that the composite in question is the inclusion of the subgroup generated by the product of the first and third or the second and third, according to the value of $a$. The conclusion follows as before by applying Corollary \ref{cor:kernel} to an appropriate homomorphism of free groups.
\end{proof}

\section{Toric and detection homomorphisms}

In this section, we define homomorphisms into and out of the fundamental group of $B_k(\graf)$. Understanding the interactions among these homomorphisms will be the key step in the proof of the main results. 

\begin{remark}
The ideas here are very closely related to those developed in \cite[\S 5]{Knudsen:TCPGBGSM} in the setting of ordered configuration spaces. We direct the reader to \cite[\S 4.3]{Knudsen:FCB} for a side-by-side comparison within a common framework.
\end{remark}


\begin{definition}
Fix a graph $\graf$ and $k\geq0$. We define the \emph{detection} homomorphism \[\delta: \pi_1(B_k(\graf))\to \bigsqcap_v\pi_1(B_k(\graphfont{\Lambda}(v))\] to be the homomorphism with $v$th component $(q_v)_*$, where the product is taken over the set of essential vertices of $\graf$.
\end{definition}

Note that, by Proposition \ref{prop:local graph}, the target of the detection homomorphism is a product of free groups.

Next, we will define two ``toric'' homomorphisms \emph{into} the fundamental group. We begin by subdividing the (connected) graph $\graf$ to guarantee that $\graf$ has no self-loops or multiple edges, that every vertex adjacent to an essential vertex has valence $2$, and that distinct essential vertices have disjoint closed stars. We fix a parametrization of each edge of $\graf$ and an ordering of the set of edges at each essential vertex. If the vertex is separating, we require the first two edges in the ordering to lie in different components of the complement. We work with a basepoint $x_0\in\Conf_k(\graf)$ chosen so that every coordinate is a bivalent vertex, which is always achievable after subdivision.

Next, we fix a set $W_0$ of vertices of valence at least $4$, a set $W_1$ of separating trivalent vertices, and a set $W_2$ of non-separating trivalent vertices, and write $W=W_0\sqcup W_1\sqcup W_2$. For $v\in W$, define \[k(v)=\begin{cases}
3&\quad v\in W_1\\
2&\quad \text{otherwise}
\end{cases}\] and set $k(W)=\sum_{v\in W}k(v)$.

\begin{construction}
For $v\in W$ and $a\in\{0,1\}$, define $\varphi^a_v:S^1\to B_{k(v)}(\stargraph{d(v)})$ as follows.
\begin{enumerate}
\item If $v\in W_0$, then $\varphi^a_v$ is the composite $S^1\xrightarrow{\epsilon} B_2(\stargraph{3})\xrightarrow{\iota_a} B_2(\stargraph{d(v)})$.
\item If $v\in W_1$, then $\varphi^a_v$ is the composite $S^1\xrightarrow{\epsilon} B_2(\stargraph{3})\xrightarrow{\sigma_{e_{a+1}}}B_3(\stargraph{3})$.
\item If $v\in W_2$, then $\varphi_v^0=\epsilon$ and $\varphi_v^1$ is the constant map with value the configuration $\{v_1,v_2\}\in B_2(\stargraph{3})$.
\end{enumerate}
\end{construction}

For each vertex $v\in W$, there is a unique piecewise linear cellular embedding $\stargraph{d(v)}\to \graf$ preserving the ordering of edges, and these embeddings give rise to an embedding from a $W$-indexed disjoint union of star graphs (we use our assumption on the subdivision of $\graf$). For $a\in\{0,1\}$, we obtain the composite embedding
{\small\[\varphi_W^a:(S^1)^W\xrightarrow{(\varphi_v^a)_{v\in W}} \bigsqcap_{v\in W}B_{k(v)}\left(\stargraph{d(v)}\right)\to B_{k(W)}\left(\bigsqcup_{v\in W}\stargraph{d(v)}\right)\to B_{k(W)}(\graf),\]}where the second map is the homeomorphism onto the appropriate connected component.

Finally, we choose a path $\alpha^a_W$ in $B_{k(W)}(\graf)$ from $\varphi_W^a(1)$ to $x_0$ for each $a\in\{0,1\}$. We require these paths to be concatenations of coordinatewise linear edge paths in $\graf$ (after choosing parametrizations) with all but one particle stationary. Such a path exists by our assumption that $\graf$ is connected and $m(\graf)>0$ (in particular, $B_{k(W)}(\graf)$ is connected).

\begin{definition}
Let $W$ be a set of essential vertices. For $a\in\{0,1\}$, we define the \emph{toric} homomorphism $\tau^a_W$ to be the composite \[\mathbb{Z}^W\cong  \pi_1(S^1,1)^W\xrightarrow{(\varphi^a_W)_*}\pi_1(B_k(\graf), \varphi_W^a(1))\xrightarrow{\hat \alpha_W^a}\pi_1(B_k(\graf),x_0).\]
\end{definition}

\begin{remark}
We now explain how these homomorphisms depend on the choices made in their construction (none of which are reflected in our notation). Changing the subdivision of $\graf$ alters $\varphi_v^a$ by a homotopy, so the final homomorphism is unchanged. Changing the ordering of the edges at an essential vertex results in a genuinely different homomorphism, since it amounts to a change of basis in the free group $\pi_1(\stargraph{d(v)})$. Finally, the choice of basepoint and path affect the definition only up to conjugation in the target.
\end{remark}

\section{The main result}

Throughout this section, we retain the notation introduced above. The key result concerning the interaction of the toric and detection homomorphisms is the following.

\begin{proposition}\label{prop:detection}
Let $W$ be a set of essential vertices as above.
\begin{enumerate}
\item The composite $\delta\circ \tau_W^0$ is injective.
\item The restrictions of $\delta\circ\tau_W^1$ to $\mathbb{Z}^{W_0\sqcup W_1}$ and $\mathbb{Z}^{W_2}$ are injective and trivial, respectively.
\item For $a\in\{0,1\}$, the homomorphisms $\delta\circ\tau_W^a$ have disjoint conjugates.
\end{enumerate}
\end{proposition}

Before proving this result, we show how it implies the main theorems.

\begin{proof}[Proof of Theorem \ref{thm:estimate}]
Choosing $W$ as above and any edge $e$, consider the commuting diagram of homomorphisms
\[\xymatrix{
&\pi_1(B_k(\graf))\ar[r]^-\delta& \displaystyle\bigsqcap_v\pi_1(B_k(\graphfont{\Lambda}(v),q_v(x_0))\\
\mathbb{Z}^W\ar@{-->}[ur]^-{\tilde\tau_W^a}\ar[r]_-{\tau_W^a}&\pi_1(B_{k(W)}(\graf))\ar[u]_-{(\sigma_e)_*^{k-k(W)}}\ar[r]^-\delta&\displaystyle\bigsqcap_v\pi_1(B_{k(W)}(\graphfont{\Lambda}(v),q_v(x_0)),\ar[u]
}\]where the righthand vertical map is a product of homomorphisms induced by sink stabilization maps, and the dashed homomorphism is determined by requiring that the triangle commute. From Corollary \ref{cor:homotopy retracts}(4), Lemma \ref{lem:disjointness homomorphism}(1), and Proposition \ref{prop:detection}, it follows that the homomorphisms $\tilde \tau_W^a$ have disjoint conjugates for $a\in\{0,1\}$. Setting $G=\pi_1(B_k(\graf))$, we obtain the chain of inequalities
\begin{align*}
\mathrm{TC}_r(B_{k}(\graf))&=\mathrm{TC}_r(BG)\\
&\geq \mathrm{cd}\left(\mathrm{im}(\tilde \tau_W^0)\times\mathrm{im}(\tilde \tau_W^1)\times G^{r-2})\right)\\
&=\mathrm{cd}\left(\mathbb{Z}^{W}\times\mathbb{Z}^{W_0\sqcup W_1}\times G^{r-2}\right)\\
&\geq\mathrm{cd}_\mathbb{Q}\left(\mathbb{Z}^{W}\times\mathbb{Z}^{W_0\sqcup W_1}\times G^{r-2}\right)
\end{align*}
where the first uses Theorems \ref{thm:asphericity} and the homotopy invariance of $\mathrm{TC}_r$ \cite{BasabeGonzalezRudyakTamaki:HTCS}, the second uses Theorem \ref{thm:lower bound} and third claim of Proposition \ref{prop:detection}, the third uses the first two claims of Proposition \ref{prop:detection}, and the fourth follows by extension of scalars. Setting $c_i=|W_i|$, the claim now follows from the rational K\"{u}nneth isomorphism and the well-known fact that the rational (co)homology of $B_k(\graf)$ is nonzero in degree $\min\left\{\left\lfloor\frac{k}{2}\right\rfloor,\,m(\graf)\right\}$ (see \cite[Lem. 3.18]{AnDrummondColeKnudsen:ESHGBG}, for example).
\end{proof}

\begin{proof}[Proof of Theorem \ref{thm:stability}] For $k\geq 2m(\graf)$, the upper bound $\mathrm{TC}_r(B_k(\graf))\leq rm(\graf)$ is well known; indeed, the configuration space deformation retracts onto a cubical complex of dimension $m(\graf)$ \cite{Swiatkowski:EHDCSG}, so the desired bound follows from standard facts about $\mathrm{TC}_r$ \cite{Rudyak:HATC}. Our argument will establish the matching lower bound.

Assuming that $r\geq2$, we invoke Theorem \ref{thm:estimate} with each $c_i$ maximal and $k$ large enough so that the theorem applies. By our assumption on $\graf$, we have $c_2=0$, so $c_0+c_1=m(\graf)$. We obtain the string of inequalities
\begin{align*}
\mathrm{TC}_r(B_k(\graf))&\geq (r-2)\min\left\{\left\lfloor\frac{k}{2}\right\rfloor,\,m(\graf)\right\}+2(c_0+c_1)+c_2\\
&\geq (r-2)\min\left\{\left\lfloor\frac{2(c_0+c_2)+3c_1}{2}\right\rfloor,\,m(\graf)\right\}+2(c_0+c_1)+c_2\\
&= (r-2)\min\left\{\left\lfloor m(\graf)+\frac{c_1}{2}\right\rfloor,\,m(\graf)\right\}+2m(\graf)\\
&=rm(\graf),
\end{align*} where the first is the conclusion of the theorem, the second is our assumption on $k$, and the third is our assumption on $\graf$.

The case $r=1$ is well known to experts; for example, it follows by combining Theorem \ref{thm:asphericity}, the Eilenberg--Ganea theorem \cite{EilenbergGanea:OLSCAG}, and the same fact about rational (co)homology.
\end{proof}

The pith of Proposition \ref{prop:detection} is contained in the following two lemmas.

\begin{lemma}\label{lem:factors split}
Let $W$ be a set of essential vertices as above. For $w\in W$, the composite $(q_w)_*\circ \tau_W^a$ is trivial on the factor of $\mathbb{Z}^W$ indexed by $v\in W$ if either $v\neq w$ or $a=1$ and $v\in W_2$.
\end{lemma}
\begin{proof}
By construction, if $a=1$ and $v\in W_2$, then the restriction of $\tau^a_W$ to the factor in question factors through the homomorphism induced by $\varphi_v^1$, which is a constant map. In the case $v\neq w$, the argument is essentially identical to that of \cite[Lem. 6.1]{Knudsen:TCPGBGSM}, the essential point being that the image of the embedding $\varphi_v^a:\stargraph{d(v)}\to \graf$ is contained in the closed star of $v$, hence disjoint from the open star of $w$ by our assumption on the subdivision of $\graf$, whence it follows that the composite $q_w\circ\varphi_v^a$ is constant.
\end{proof}

\begin{lemma}\label{lem:factors disjoint}
Fix $v\in W$. For $a\in\{0,1\}$ the following composite homomorphisms have disjoint conjugates:
\[\mathbb{Z}\xrightarrow{v\in W} \mathbb{Z}^W\xrightarrow{(\varphi_W^a)_*} \pi_1(B_{k(W)}(\graf))\xrightarrow{(q_v)_*}\pi_1(B_{k(W)}(\graphfont{\Lambda}(v))).\] Moreover, each composite is injective provided $a=0$ or $v\notin W_2$.
\end{lemma}
\begin{proof}
We treat the disjointness claim first. Suppose first that $v\in W_0$. We have the commutative diagram of spaces
{\small\[
\xymatrix{
S^1\ar[ddd]_-\epsilon\ar[r]^-{v\in W}&(S^1)^W\ar[d]_-{\bigsqcap_{w\in W}\varphi_v^a}\ar[r]^-{\varphi_W^a}&B_{k(W)}(\graf)\ar[r]^-{q_v}&B_{k(W)}(\graphfont{\Lambda}(v))\ar[d]\\
&\bigsqcap_{w\in W}B_{k(w)}\left(\stargraph{d(w)}\right)\ar[r]^-\subseteq&B_{k(W)}\left(\bigsqcup_{w\in W}\stargraph{d(w)}\right)\ar[u]&B_{k(W)}(\stargraph{d(v)}/\partial)\\
&&&B_{2}(\stargraph{d(v)}/\partial)\ar[u]\\
B_2(\stargraph{3})\ar[r]^-{\iota_a}&B_2(\stargraph{d(v)})\ar[uu]^-{v\in W}\ar[rr]&&B_2(\graphfont{\Lambda}(v))\ar[u]&\\
}
\]}where the unmarked horizontal maps are either induced maps or sink stabilization. For commutativity, we again use that the image of $\varphi_w^a$ is disjoint from the open star of $v$ for $w\neq v$. According to Lemma \ref{lem:local calculation}, each of the resulting composite maps $S^1\to B_{k(W)}(\stargraph{d(v)}/\partial)$  induces an injection on fundamental groups. Thus, by Lemma \ref{lem:disjointness homomorphism}(1), it suffices to show that these induced homomorphisms have disjoint conjugates for $a\in\{0,1\}$, but this was proven in Lemma \ref{lem:local calculation} by commutativity. The case of $v\in W_1$ is essentially the same, with Lemma \ref{lem:other local calculation} playing the role of Lemma \ref{lem:local calculation} (here we use our assumption on the orderings of edges at separating vertices). For $v\in W_2$, there is nothing to prove, since the homomorphism for $a=1$ is trivial in this case.

The injectivity claim is essentially immediate from the injectivity clauses of Lemmas \ref{lem:local calculation} and and \ref{lem:other local calculation}.
\end{proof}

\begin{proof}[Proof of Proposition \ref{prop:detection}]
By Lemma \ref{lem:factors split}, the homomorphism $\delta\circ \tau_W^a$ splits as a product of homomorphisms indexed by $W$. The same result implies that the factor indexed by $v\in W_2$ is trivial if $a=1$. Thus, it suffices to show that the factors indexed by $v\in W$ have disjoint conjugates for $a\in\{0,1\}$; and that they are injective if either $a=0$ or $v\notin W_2$. Since the vertical homomorphisms are isomorphisms in the commutative diagram
\[
\xymatrix{
&&\pi_1(B_{k(W)}(\graf),x_0)\ar[r]^-{(q_v)_*}&\pi_1(B_{k(W)}(\graphfont{\Lambda}(v)), q_v(x_0))\\
\mathbb{Z}\ar[r]^-{v\in W}&\mathbb{Z}^W\ar[ur]^-{\tau_W^a}\ar[r]_-{(\varphi_W^a)_*}&\pi_1(B_{k(W)}(\graf),\varphi_W^a(1))\ar[u]_-{\hat\alpha_W^a}\ar[r]^-{(q_v)_*}&\pi_1(B_{k(W)}(\graphfont{\Lambda}(v)), q_v(\varphi_W^a(1)))\ar[u]_-{\hat\beta_W^a}
}
\] where $\beta=q_v\circ\alpha_W^a$, this claim is equivalent to Lemma \ref{lem:factors disjoint}.
\end{proof}

\bibliographystyle{amsalpha}
\providecommand{\bysame}{\leavevmode\hbox to3em{\hrulefill}\thinspace}
\providecommand{\MR}{\relax\ifhmode\unskip\space\fi MR }
\providecommand{\MRhref}[2]{%
  \href{http://www.ams.org/mathscinet-getitem?mr=#1}{#2}
}
\providecommand{\href}[2]{#2}

\end{document}